\documentclass[12pt]{amsart} 
\usepackage{amsmath,amssymb,latexsym,amsthm,enumerate,amscd,hyperref,enumitem} 
\usepackage[alphabetic]{amsrefs} 
\usepackage{graphicx}
\usepackage{pinlabel}
\usepackage[dvipsnames]{xcolor}
\usepackage{graphicx,tikz, tikz-cd}
\usepackage{fullpage}

\makeatletter
\@namedef{subjclassname@2020}{%
  \textup{2020} Mathematics Subject Classification}
\makeatother

\theoremstyle{definition}
\newtheorem{thm}{Theorem}
\newtheorem{lem}[thm]{Lemma}
\newtheorem{prop}[thm]{Proposition}

\newtheorem{cor}[thm]{Corollary}

\newtheorem{conj}[thm]{Conjecture}
\newcounter{theoremalph}

\DeclareMathOperator{\cd}{cd}

\DeclareMathOperator{\Mod}{Mod}

\begin{document}
\title{The $K(\pi,1)$-conjecture implies the center conjecture for Artin groups}
\author{Kasia Jankiewicz}  
\author{Kevin Schreve}
\address{Department of Mathematics, University of California, Santa Cruz, CA 95064}
\email{kasia@ucsc.edu}
\address{Department of Mathematics, Louisiana State University, Baton Rouge, LA 70806}
\email{kschreve@lsu.edu}
\thanks{The first author was supported by the NSF grant DMS-2203307}
\date{\today}
\maketitle

 \begin{abstract}
 In this note, we prove that the $K(\pi,1)$-conjecture for Artin groups implies the center conjecture for Artin groups. Specifically, every Artin group without a spherical factor that satisfies the $K(\pi,1)$-conjecture has a trivial center.
 \end{abstract}

\section{Introduction}
A Coxeter system $(W,S)$ consists of a group $W$ and a generating set $S$ where $W$ is given by a presentation 
$$W = \langle s \in S| s^2 = (st)^{m_{st}} = 1 \rangle,$$ where $m_{st} \in \{2,3,\dots\} \cup \{\infty\}$. 
The associated Artin group $A$ is given by the presentation 
$$A = \langle s \in S| \underbrace{sts\cdots}_{m_{st} \hspace{.5mm} terms} = \underbrace{tst\cdots}_{m_{st} \hspace{.5mm} terms} \rangle.$$ 
An Artin group $A$ is \emph{spherical} if the corresponding Coxeter group is finite, and otherwise $A$ is \emph{infinite type}.
The Coxeter diagram $\Gamma_S$ is a graph with vertices corresponding to $S$ and where two vertices are joined by an edge if and only if $m_{st}>2$. If $m_{st}\geq 4$ we label the edge with $m_{st}$.
A \emph{special subgroup} of $A$ is a subgroup generated by some subset of $S$. Each special subgroup is itself an Artin group \cite{vanderLek83}.
Each Artin group with standard generating set admits (a possibly trivial) decomposition $A= A_{T_1}\times \dots \times A_{T_n}$ where each $T_i\subseteq S$ defines a connected component of the Coxeter graph $\Gamma_S$.
An Artin group $A$ is \emph{irreducible} if its Coxeter diagram is connected.
We say $A_{T_i}$ is a \emph{spherical factor} of $A$ if $A_{T_i}$ is spherical.
Every irreducible spherical Artin group has an infinite cyclic center \cite{Deligne72, BrieskornSaito72}. Conjecturally, those are the only irreducible Artin groups with nontrivial center.
\begin{conj}[The Center Conjecture]
Every Artin group without a spherical factor has trivial center.
\end{conj}
 
 The center conjecture holds for FC-type Artin groups and $2$-dimensional Artin groups \cite{GodelleParis2012}. 
Charney and Morris-Wright have shown the center conjecture holds for Artin groups whose defining graphs are not stars of a single vertex \cite{CharneyMorrisWright2019}.
Godelle and Paris further showed that if all Artin groups with $m_{st} \ne \infty$ for all $s,t \in S$ satisfy the center conjecture, then all Artin groups satisfy the center conjecture \cite{GodelleParis2012}.

The FC-type and  $2$-dimensional Artin groups also satisfy the $K(\pi,1)$-conjecture \cite{CharneyDavisCAT0}.
\begin{conj}[The $K(\pi, 1)$-conjecture]
The orbit space $\mathcal H(W)/W$ of a complexified hyperplane arrangement associated to a Coxeter system $(W,S)$ is a $K(\pi,1)$ for the Artin group $A$ associated to $W$.
\end{conj}
For a precise definition of $\mathcal H(W)$ and more background, see e.g. the survey paper \cite{Paris2014}.
It is known that the fundamental group of $\mathcal H(W)/W$ is equal to $A$, so the conjecture is about the asphericity of $\mathcal H(W)/W$. 
In this note, we prove the following:

\begin{thm}
\label{t:main} 
Every Artin group without a spherical factor that satisfies the $K(\pi, 1)$-conjecture has trivial center.
\end{thm}

In fact, we only need the following consequence of the $K(\pi, 1)$-conjecture: an Artin group $A$ which satisfies the $K(\pi,1)$-conjecture has finite cohomological dimension which is realized by a spherical subgroup, i.e. $\cd(A) = \cd(A_T) = |T|$ for some spherical subset $T$. See Theorem~\ref{thm:main general} for the more general statement of our main theorem.

For example, in \cite{McCammondWinterBraids} McCammond mentions that the center conjecture is unknown for the Artin group $A$ with the Coxeter diagram as in Figure~\ref{fig:smallest mysterious}.
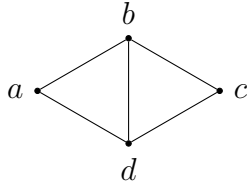
\begin{figure}
\begin{tikzpicture}[scale = 0.7]
\node[circle, draw, fill, inner sep = 0pt,minimum width = 2pt, label = below:$d$] (a) at (0,-1) {};
\node[circle, draw, fill, inner sep = 0pt,minimum width = 2pt, label = above:$b$] (b) at (0,1) {};
\node[circle, draw, fill, inner sep = 0pt,minimum width = 2pt, label = left:$a$] (c) at (-1.73,0) {};
\node[circle, draw, fill, inner sep = 0pt,minimum width = 2pt, label = right:$c$] (d) at (1.73,0) {};
\draw (a) -- (c) -- (b) -- (d) -- (a) -- (b);
\end{tikzpicture}
\caption{The Coxeter diagram of the smallest mysterious Artin group according to McCammond.}\label{fig:smallest mysterious}
\end{figure}
The group $A$ is given by the presentation
$$\langle a,b,c,d \mid aba = bab, bcb=cbc, cdc = dcd, dad = ada, bdb = dbd, ac = ca\rangle.$$
The $K(\pi, 1)$-conjecture holds for $A$ by a theorem of Charney \cite{Charney2004}. That allows us to answer the question of McCammond.

\begin{cor}
The Artin group with the Coxeter diagram as in Figure~\ref{fig:smallest mysterious} has trivial center.
\end{cor}

Another class of Artin groups which satisfy the $K(\pi,1)$-conjecture are the \emph{locally reducible} Artin groups, where all irreducible spherical subgroups are of rank $\le 2$ \cite{Charney2000}. There are many of these with $m_{st} \ne \infty$ for each $s,t \in S$, and as far as we know the center conjecture was open here.

\begin{cor}
Every locally reducible Artin group without spherical factor has trivial center.
\end{cor}


\section{Representations of Artin groups in mapping class groups}
An Artin group $A$ with standard generating set $S$ has \emph{small type} if $m_{st}\in\{2,3\}$ for all $s,t\in S$. 
In this section we recall a representation of small type Artin groups in mapping class groups, due to Crisp-Paris \cite{CrispParis2001}, and analyze where certain elements of $A$ are mapped.

Let $\Sigma$ be a surface. 
A \emph{multicurve} is a collection of pairwise disjoint simple closed curves on $\Sigma$. 
We say that two multicurves are \emph{disjoint} if their isotopy classes have disjoint representatives. 
A \emph{multitwist} about a multicurve $\gamma$ is a product of non-trivial powers of Dehn twists about simple closed curves in $\gamma$ (the powers can be different for different curves).
We shall need the following lemma about commuting multitwists about multicurves.
\begin{lem}\label{lem:disjoint multicurves}
Let $\gamma$ and $\gamma'$ be essential multicurves on $\Sigma$ that do not share a simple closed curve. Let $T_\gamma$ and $T_{\gamma'}$ be the associated multitwists along $\gamma$ and $\gamma'$. Then $T_\gamma$ and $T_{\gamma'}$ commute if and only $\gamma$ and $\gamma'$ are disjoint.
\end{lem}

\begin{proof}
The if direction is clear, so suppose $T_\gamma$ and $T_{\gamma'}$ commute. Then $T_{\gamma}^N$ and $T_{\gamma'}^N$ commute for all $N$. By a theorem of Koberda \cite{Koberda12}, the group generated by large powers of Dehn twists of all curves in $\gamma \cup \gamma'$ is a right-angled Artin group. Therefore, we have two words $w$ and $w'$ in a RAAG, where $w$ and $w'$ are nontrivial powers of commuting generators. It follows from the normal form for RAAG's that these commute exactly when each generator in $w$ commutes with each generator of $w'$ \cite{HermillerMeier95}. Therefore, the curves in $\gamma$ can be isotoped to be disjoint from the curves in $\gamma'$. 
\end{proof}

Let $A_S$ be a small type Artin group with standard generating set $S$. 
We can build an associated surface
$$\Sigma_S = \bigcup_{s\in S} An_{s}$$
 where each $An_s$ is an annulus. 
We denote the meridian of $An_s$ by $\gamma_s$. 
If $m_{st} = 3$, then we arrange that the annuli $An_s, An_t$ intersect in a single square so that $\gamma_s, \gamma_t$ intersect transversely at one point, and any triple intersection of annuli is empty. 
\begin{figure}
\begin{tikzpicture}
\begin{scope}[shift={(0,1)}]
\node[circle, draw, fill, inner sep = 0pt,minimum width = 2pt, color=NavyBlue] (a) at (0,0) {};
\node[circle, draw, fill, inner sep = 0pt,minimum width = 2pt, color=red] (b) at (0:1) {};
\node[circle, draw, fill, inner sep = 0pt,minimum width = 2pt, color=ForestGreen] (c) at (60:1) {};
\node[circle, draw, fill, inner sep = 0pt,minimum width = 2pt, color=Dandelion] (d) at (120:1) {};
\draw (a) -- (c) -- (d) -- (a) -- (b);
\end{scope}
\begin{scope}[shift={(4,0)}, scale = 0.7]
\draw (0.5,0) -- (10,0);
\draw (10,1) -- (0.5,1);
\draw[densely dotted, semithick] (0.5,0) -- (0.5,1);
\draw[densely dotted, semithick] (10,0) -- (10,1);
\draw[NavyBlue, thick ] (0.5, 0.5) -- (10,0.5); 
\draw[densely dotted, semithick] (1,-1) -- (1+0.866, -1.5);
\draw (1+0.866, -1.5) -- (4.5+0.866, -1.5+7*0.866);
\draw[densely dotted, semithick] (4.5+0.866, -1.5+7*0.866) -- (4.5, -1+7*0.866);
\draw (4.5, -1+7*0.866) -- (1,-1);
\draw[Dandelion, thick ] (1+0.433, -1.25) -- (4.5+0.433, -1.25 +7*0.866);
\draw[densely dotted, semithick] (7,-1) -- (7-0.866, -1.5);
\draw (7-0.866, -1.5) -- (3.5-0.866, -1.5 + 7*0.866);
\draw[densely dotted, semithick] (3.5-0.866, -1.5 + 7*0.866) -- (3.5, -1 +7*0.866);
\draw (3.5, -1 +7*0.866) -- (7,-1);
\draw[ForestGreen, thick ] (7-0.433, -1.25) -- (3.5-0.433, -1.25 +7*0.866);
\draw[densely dotted, semithick] (8,-1) -- (9,-1);
\draw (9,-1) -- (9,5);
\draw[densely dotted, semithick] (9,5) -- (8,5);
\draw (8,5) -- (8,-1);
\draw[red, thick] (8.5, -1) -- (8.5, 5); 
\end{scope}
\end{tikzpicture}
\caption{An example of the surface $\Sigma_S$ (right) corresponding to the Artin group $A$ with the standard generating set $S$ whose Coxeter graph $\Gamma$ is illustrated (left). For each vertex $s$ in $\Gamma$, there is a corresponding rectangle in the right picture, which is glued along the dotted sides to an annulus $An_s$. Its meridian $\gamma_s$ has the same color as $s$.}\label{fig:surface}
\end{figure}
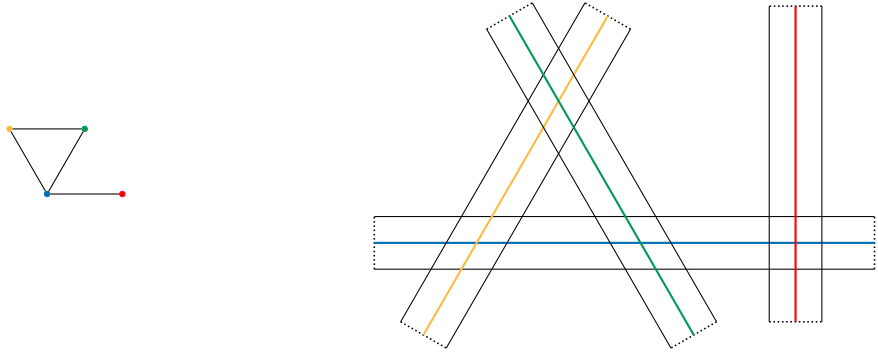
If $m_{st} = 2$, then $An_s, An_t$ will be  disjoint.
See Figure~\ref{fig:surface} for an example.
There is a representation of $A_S$ in the mapping class group of $\Sigma_S$ where each generator $s\in S$ is mapped to the Dehn twist about the simple closed curve $\gamma_s$ (see Proposition~\ref{prop:small type rep}).
Full details can be found in \cite{CrispParis2001}. 
Each subset $T$ of $S$ has an associated subsurface $\Sigma_T$ of $\Sigma_A$. 
By construction, the subsurface $\Sigma_T$ for an irreducible spherical subset $T\subseteq S$ and the induced homomorphism $\rho_T:A_T\to \Mod(\Sigma_T)$ are exactly the Perron-Vannier representation of small type spherical Artin groups \cite{PerronVannier1996} (see also \cite{JankiewiczSchreveRAAGsubgroupsArtin}).
For every spherical subset $T$, the generator of the center of the group $A_T$ is denoted by $z_T$.  
By \cite[Prop 2.12]{LabruereParis01} $\rho(z_T^4)$ is a multitwist about the boundary of $\Sigma_T$. We denote that multicurve by $\Gamma_T$.

Let $K_S$ be the graph that is the union $\bigcup_{s\in S} \gamma_s\subseteq \Sigma_S$. 
By construction, there is a deformation retraction $r: \Sigma_S \to K_S$. 
Thus $H_1(\Sigma_S) = H_1(K_S)$, and in particular if $S$ is spherical then $H_1(\Sigma_S) = \mathbb Z^S$. For any $S'\subset S$, the map $H_1(\Sigma_{S'})\to H_1(\Sigma_S)$ induced by the inclusion $\Sigma_{S'}\hookrightarrow \Sigma_S$ is injective and $H_1(\Sigma_{S'})\subsetneq H_1(\Sigma_S)$.

Every closed path in a graph is homotopic to a cycle (i.e.\ a closed path without backtracks). 
In particular, every homotopy class of a simple closed curve in $\Sigma_S$ can be realized as a cycle in $K_S$. 
We will now view all the simple closed curves in $\Sigma_S$ as cycles in $K_S$.
In particular, we view components of $\gamma_T$ for any spherical subset $T\subseteq S$ as cycles in $K_S$.

\begin{lem}\label{lem:intersecting boundary}
Let $T\subseteq S$ be an irreducible spherical subset, and $s\in S-T$. Then $\gamma_s$ intersects $\gamma_T$ if and only if $\gamma_s$ intersects $\gamma_t$ for some $t\in T$.
\end{lem}
\begin{proof} If $\gamma_s$ does not intersect any $\gamma_t$ for $t\in T$, then $\rho(s)$ commutes with $\rho(t)$ for all $t \in T$. 
Then $\rho(s)$ must also commute with $\rho(z_T^2)$, and by Lemma~\ref{lem:disjoint multicurves} $\gamma_s$ and $\gamma_T$ are disjoint.

Now suppose that $\gamma_s$ intersects $\gamma_t$ for some $t\in T$. By construction $\gamma_s, \gamma_t$ intersect exactly once. Suppose $\gamma_s$ can be isotoped in $\Sigma_S$ to be disjoint from $\partial \Sigma_T$. Then $\gamma_s\subseteq \Sigma_T$, since $\gamma_s, \gamma_t$ still must intersect. In particular, $[\gamma_s]\in H_1(\Sigma_T) = H_1(K_T)$. This is a contradiction, since $H_1(K_T)\subsetneq H_1(K_{T\cup\{s\}})$.
\end{proof}

\begin{lem}\label{lem:disjoint sphericals}
Let $T_1, T_2\subseteq S$ be two disjoint, irreducible, spherical subsets. Then any component of $\gamma_{T_1}$ and any component of $\gamma_{T_2}$ are non-isotopic and disjoint in $\Sigma_S$. 
\end{lem}
\begin{proof}
Since $T_1, T_2\subseteq S$ are disjoint, by construction, the subgraphs $K_{T_1}, K_{T_2}\subseteq K_S$ are disjoint. 
Every connected component of $\gamma_{T_i}$ can be realized as a cycle contained in $K_{T_i}$. Any two disjoint cycles 
in a graph are non-isotopic. The conclusion follows.
\end{proof}

\begin{prop}[{\cite{CrispParis2001}}]\label{prop:small type rep}
For every small type Artin group $A$ with the standard generating set $S$, there exists a surface with boundary $\Sigma_S$ and a homomorphism $\rho:A\to \Mod(\Sigma_S)$ where 
\begin{enumerate}[label=(\alph*)]
\item for each $s\in S$, $\rho(s)$ is the Dehn twist about a simple closed curve $\gamma_s$,
\item the simple closed curves $\gamma_s, \gamma_t$ are disjoint $\iff$ $m_{st} = 2$,
\item the simple closed curves $\gamma_s, \gamma_t$ intersect exactly once $\iff$ $m_{st} = 3$.
\end{enumerate}
Moreover, 
\begin{enumerate}[resume*]
\item for every irreducible spherical subset $T\subseteq S$, $\rho(z_T^2)$ is the multitwist about a multicurve $\gamma_T$ which is the boundary of the subsurface $\Sigma_T$, and
\item for every irreducible spherical $T\subseteq S$ and $s\in S-T$, the simple closed curve $\gamma_s$ and the multicurve $\Gamma_T$ are disjoint if and only if $[s,t] = 1$ for all $t\in T$.
\end{enumerate}
\end{prop}
\begin{proof}
The fact that $\rho$ is a homomorphism follows from standard relations between Dehn twists, see \cite[Prop 4]{CrispParis2001}. The parts (a), (b), (c) follow from \cite{CrispParis2001} as well.
Part (d) follows from \cite{LabruereParis01} (see discussion above). Finally part (e) is a consequence of Lemma~\ref{lem:intersecting boundary} and Lemma~\ref{lem:disjoint multicurves}. 
\end{proof}

Let $A$ be an Artin group with standard generating set $S$. 
We say $A$ is \emph{free-of-infinity} if $m_{st}<\infty$ for all $s,t\in S$.

\begin{prop}[{\cite{CrispParis2001}}]\label{prop:no infinity rep}
Let $A$ be a free-of-infinity Artin group. Then there exists a small type Artin group $\widetilde A$ with standard generating set $\widetilde S$ and a homomorphism $\phi: A\to \widetilde A$ such that
\begin{itemize}
\item there exists a partition $\displaystyle \bigsqcup_{s\in S} I(s)$ of $\widetilde S$ such that the elements of $I(s)$ pairwise commute and $\displaystyle \phi(s) = \prod_{r\in I(s)} r$,
\item $m_{st} = 2$ if and only if every element of $I(s)$ and every element of $I(t)$ commute, and
\item  if $m_{st} \geq 3$ then the subgroup generated by $I(s)\cup I(t)$ is a direct product of braid groups on $m_{st}$ strands.
\end{itemize}
Let $\rho\circ \phi:A\to \Mod(\Sigma_{\widetilde A})$ be the composition of the homomorphism $\phi$ with the homomorphism $\rho:\widetilde A\to \Mod(\Sigma_{\widetilde A})$ from Proposition~\ref{prop:small type rep}. Then 
\begin{enumerate}[label=(\alph*)]
\item for each $s\in S$,  $\rho\circ \phi(s)$ is a multitwist about a multicurve $\gamma_s = \bigcup_{r\in I(s)} \gamma_r $, 
\item $m_{st} = 2$ if and only if every component of $\gamma_s$ and every component of $\gamma_t$ are disjoint,
\item for every spherical subset $T\subset S$, $\rho(z_T^2)$ is the multitwists about a multicurve $\gamma_T$, and
\item for every $T\subseteq S$ and $s\in S-T$, the multicurve $\gamma_s$ and the multicurve $\gamma_T$ are disjoint if and only if $[s,t] = 1$ for all $t\in T$.
\end{enumerate}
\end{prop}
\begin{proof}
The homomorphism $\phi$ is described in \cite{Crisp99} and also in \cite{CrispParis2001}. Parts (a) and (b) follow directly from the construction. Part (c) is proven in \cite[Lem 6.1]{JankiewiczSchreveRAAGsubgroupsArtin}. Part (d) follows from Lemma~\ref{lem:intersecting boundary} and Lemma~\ref{lem:disjoint multicurves}.
\end{proof}

\section{The main theorem}
We will need the following lemma.
\begin{lem}\label{lem:cd}
Let $A_S$ be an Artin group which splits as a product $A_S = A_{U}\times A_{V}$ where $A_U$ is the maximal spherical factor. Suppose that $\cd A_S<\infty$. Then $\cd A_S = \cd A_U + \cd A_V$.
\end{lem}
\begin{proof} By \cite[Thm 5.5]{BieriBook81} a group $G=N\times Q$ has $\cd G = \cd N + \cd Q$ provided that 
\begin{itemize}
\item $\cd Q< \infty$, and
\item $N$ is of type FP and $H^n(N, \mathbb Z N)$ is free for $n= \cd N$.
\end{itemize}
Clearly $\cd A_V<\infty$ since $\cd A_S<\infty$. Since $A_U$ is a spherical Artin group, $A_U$ has type $FP$. By \cite[Thm B]{Squier94} (see also \cite{Bestvina99}) $A_U$ is a duality group, so $H^n(A_U, \mathbb Z A_U)$ is free. The conclusion follows.
\end{proof}

\begin{thm}\label{thm:main general}
Let $A_S$ be an Artin group of infinite type with the standard generating set $S$ such that $A_S$ has no spherical factors. 
If $\cd A = \cd A_T= |T|$ for some spherical subset $T\subseteq S$, then $A_S$ has trivial center.
In particular if $A_S$ satisfies the $K(\pi,1)$-conjecture, then $A_S$ has trivial center.
\end{thm}

\begin{proof}[Proof for free-of-infinity case]
First suppose that $A_S$ is free-of-infinity.
Let $T\subseteq S$ be a maximal spherical subset such that $\cd A_S = \cd A_T$.
Let $T_1\sqcup T_2 \sqcup \dots\sqcup T_n$ be the decomposition of $T$ into irreducible spherical subsets inducing the decomposition $A_T = A_{T_1}\times\dots\times A_{T_n}$. 
Since $A_S$ has no spherical factors for each $i=1,\dots, n$ there exists $s_i\in S-T$ such that $[s_i, z_{T_i}]\neq 1$ as otherwise $A_{T_i}$ would be a spherical factor of $A_S$.
In particular, for each $i=1, \dots, n$, there exists $t_i\in T_i$ such that $[s_i,t_i]\neq 1$.

Consider the representation of $\rho:A_S \to \Mod (\Sigma_S)$ from Proposition~\ref{prop:no infinity rep}. 
By Proposition~\ref{prop:no infinity rep}, $\rho(s_i)$ and $\rho(z_{T_i})$ are the Dehn twists about multicurves $\gamma_{s_i}$ and $\gamma_{T_i}$ respectively, where $\gamma_{s_i}$ and $\gamma_{T_i}$ intersect.

Suppose that $A_S$ has nontrivial center and let $y\in Z(A_S)$ with $y\neq e$.
Note that $y$ has infinite order since $A_S$ is torsion-free, as $\cd A_S<\infty$.
If $y^k\notin A_T$ for any $k\neq 0$, then $\langle A_T, y\rangle \simeq A_T\times \mathbb Z$ is a subgroup of $\cd A+1$, which is a contradiction.
Thus there exists $k\in \mathbb N$ such that $y^k\in A_T$. Then $y^k\in Z(A_T)$, i.e.\ $y^m = \prod_{i=1}^n z_{T_i}^{m_i}$ for some $m>0$ and at least one of $m_1,\dots, m_n$, say $m_1$, is non-zero. 
By Lemma~\ref{lem:disjoint sphericals}, $\rho(y^m)$ is a multitwist about a multicurve $\gamma = \sqcup \gamma_{T_i}$ in $\Sigma$ where the union is taken over all $i$ such that $m_i\neq 0$. 
In particular, $\gamma$ intersects $\gamma_{s_1}$. 
By Lemma~\ref{lem:disjoint multicurves}, $[\rho(y^m), \rho(s_1)]\neq 1$. 
Thus $[y, s_1]\neq 1$. 
This contradicts the fact that $y$ is a central element of $A$.
\end{proof}

\begin{proof}[Proof for general case]
The general case is induction on the cardinality of $S$.
Suppose $\cd A_S = \cd A_T$ where $T\subseteq S$ is a spherical subset. 
Suppose there exist generators $v,w\in S$ such that $m_{vw} = \infty$. 
The group $A_S$ splits as an amalgamated product $A_{S\setminus\{v\}}*_{A_{S\setminus\{v,w\}}} A_{S\setminus\{w\}}$. 
Since $T$ cannot contain both $v$ and $w$, we have $T\subseteq S\setminus \{v\}$ or $T\subseteq S\setminus\{w\}$. 
Without loss of generality we assume that $T\subseteq S\setminus \{v\}$.
It follows that $\cd A_{S\setminus\{v\}} = \cd A_T$, as $\cd A_{S\setminus\{v\}} \le \cd A_S$.
If $A_{S\setminus\{v\}}$ has no spherical factor,  then by induction $A_{S\setminus\{v\}}$ has trivial center.
By \cite[Lem 3.2]{GodelleParis2012} the center of the amalgamated product $A$ is also trivial.

Now suppose that $A_{S\setminus \{v\}}$ has a nontrivial spherical factor. 
Let $$A_{U_1}\times \dots \times A_{U_p}\times A_{V_1}\times \dots \times A_{V_q}$$ be the decomposition of $A_{S\setminus \{v\}}$ into irreducible factors where each $A_{U_i}$ is spherical and each $A_{V_j}$ has infinite type. 
Let $A_V =  A_{V_1}\times \dots \times A_{V_q}$. 
By maximality $U_i\subseteq T$ for all $i=1,\dots, p$. 
Let $T' = V\cap T$. 
Then $\cd A_{V} = \cd A_{T'}$. Indeed by Lemma~\ref{lem:cd} ,$$\cd A_V = \cd A_{S\setminus\{v\}}-\Sigma_{i=1}^p\cd A_{U_i} = \cd A_T -\Sigma_{i=1}^p\cd A_{U_i} = A_{T'}.$$ By the inductive assumption $Z(A_{V}) = \{1\}$, and thus $Z(A_{S\setminus\{v\}}) \subseteq \langle z_{U_1}\rangle\times \dots \times \langle z_{U_n}\rangle $.

 Since $A_S$ does not have a spherical factor, for every $i=1, \dots, n$ we have $[v, z_{U_i}]\neq 1$. 
In particular, each set $U_i$ contains a standard generator $u_i$ such that $m_{vu_i} \geq 3$. 
Since $Z(A_{S\setminus\{v\}})\subseteq A_T$ and by \cite[Lem 3.2]{GodelleParis2012} $Z(A)\subseteq Z(A_{S\setminus\{v\}})$, it suffices to prove that $v$ does not commute with any nontrivial element of $Z(A_T) = \langle z_{T_1}, \dots, z_{T_n}\rangle$.
By maximality of $T$, $A_{T\cup\{v\}}$ is not spherical.
By the discussion above, $A_{T\cup\{v\}}$ is irreducible, and in particular it has no spherical factors.
If $A_{T\cup\{v\}}$ is free-of-infinity, we are done.

We now assume that $A_{T\cup\{v\}}$ is not free-of-infinity.
Consider the quotient homomorphism $\phi:A_{T\cup\{v\}}\to A_{\overline  T\cup\{\overline v\}}$, where for every $t\in T$ such that $m_{tv} = \infty$ the corresponding generators $\overline t, \overline v \in \overline T\cup\{\overline v\}$ have $m_{\overline t\overline v} = 7$. 
The group $A_{\overline  T\cup\{\overline v\}}$ is irreducible. 
The only irreducible spherical Artin group containing label $7$ is the dihedral Artin group. 
If $A_{\overline  T\cup\{\overline v\}}$ is the dihedral Artin group, then $A_{T\cup\{v\}} = F_2$ and so $\cd A_{S} = 1$, i.e.\ $A_S = F(S)$. Then clearly, $A_S$ has trivial center.
Otherwise $A_{\overline  T\cup\{\overline v\}}$ is irreducible and has infinite type. Also $\cd A_{\overline  T\cup\{\overline v\}} = \cd A_{\overline T}$. 
By the free-of-infinity case, $[\overline v, \overline y]\neq 1$ for any nontrivial $\overline y\in \langle z_{\overline T_1}, \dots, z_{\overline T_n}\rangle$, as otherwise $y$ would be a central element of $A_{\overline T\cup \{\overline v\}}$. 
Thus $[v, y]\neq 1$ for any nontrivial $y\in \langle z_{T_1}, \dots, z_{T_n}\rangle$.
This completes the proof.
\end{proof}

\bibliographystyle{alpha}
\bibliography{../../kasia}

\end{document}